\documentclass[a4paper,10pt]{llncs}
\usepackage[utf8]{inputenc}

\usepackage{amsmath,amssymb}
\usepackage{enumerate,enumitem}
\usepackage{graphicx,subfigure}
\usepackage{textcomp,gensymb}
\usepackage{color,url}

\DeclareGraphicsExtensions{.pdf}
\graphicspath{{figs/}}

\newcommand{\dem}{{\rm dem}}
\newcommand{\Int}{{\rm Int}}

\newtheorem{myclaim}{Claim}
\title{Orienting triangulations\thanks{This work was supported by the project EGOS, ANR-12-JS02-002-01}}
\author{Boris Albar \and Daniel Gon\c{c}alves \and Kolja Knauer}
\institute{LIRMM, CNRS \& Universit\'e Montpellier 2}

\begin{document}

\maketitle

\begin{abstract}
We prove that any triangulation of a surface different from the sphere
and the projective plane admits an orientation without sinks such that every vertex has outdegree divisible
by three. This confirms a conjecture of
Bar\'at and Thomassen and is a step towards a generalization of
Schnyder woods to higher genus surfaces.
\end{abstract}

\section{Introduction}
The notation and results we use for graphs and surfaces can be found
in~\cite{Moh-01}. We start with some basic definitions:

A \emph{map} (or \emph{$2$-cell embedding}) of a multigraph into a
surface, is an embedding such that deleting the graph from the surface
leaves a collection of open disks, called the \emph{faces} of the
map. A \emph{triangulation} is a map of a simple graph (i.e. without
loops or multiple edges) where every face is triangular (i.e. incident to three edges).  A fundamental
result in the topology of surfaces is that every surface admits a map.
The \emph{(orientable) genus} of a map on an orientable surface is
$\frac{1}{2}(2-n+m-f)$ and the \emph{(non-orientable) genus} of a map
on a non-orientable surface is $2-n+m-f$, where $n,m,f$ denotes the
number of vertices, edges, and faces of the map, respectively. The
\emph{Euler genus} $k$ of a map is $2-n+m-f$, i.e., the non-orientable
genus or twice the orientable genus.  All the maps on a fixed surface
have the same genus, which justifies to define the \emph{(Euler) genus
  of a surface} as the (Euler) genus of any of the maps it admits.
In~\cite{Bar-06} Bar\'at and Thomassen conjectured the following:
\begin{conjecture}\label{conj}
 Let $T$ be a triangulation of a surface of Euler genus $k \geq 2$.
 Then $T$ has an orientation such that each outdegree is at least $3$,
 and divisible by $3$.
\end{conjecture} 
One easily computes that the number of edges $m$ of a triangulation $T$ of a surface of Euler genus $k$ is $3n-6+3k$.  So while triangulations of Euler genus less than $2$ simply have too few edges to satisfy the conjecture, in~\cite{Bar-06} the conjecture is proved for the case $k=2$, i.e., the torus and the Klein bottle. Moreover, they show that any triangulation $T$ of a surface has an orientation such that each outdegree is divisible by $3$, i.e, in order to prove the full conjecture they miss the property that there are no sinks.

Bar\'at and Thomassen's conjecture was originally motivated in the context of claw-decompositions of graphs, 
since given an orientation with the claimed properties the outgoing edges of each vertex can be divided into claws, such that every vertex is the center of at  least one claw. 

Another motivation for this conjecture 
is, that it can be seen as a step towards the generalization
of planar Schnyder woods to higher genus surfaces. 
A \emph{Schnyder wood}~\cite{Sch-89} of a planar triangulation is an orientation and a 
$\{0,1,2\}$-coloring of the \emph{inner} edges satisfying the following \emph{local rule} on every \emph{inner} vertex $v$: 
going counterclockwise around $v$ one successively
crosses an outgoing $0$-arc, possibly some incoming $2$-arcs, an
outgoing $1$-arc, possibly some incoming $0$-arcs, an outgoing $2$-arc, and possibly some incoming $1$-arcs until coming back to the outgoing $0$-arc. 

%
%

Schnyder woods are one of the main tools in the area of planar graph representations and Graph Drawing. They provide a machinery to construct space-efficient straight-line drawings~\cite{Sch-90,Fel-01}, representations by touching \textsf{T} shapes~\cite{Fra-94}, they yield a characterization of planar graphs via the dimension of their vertex-edge incidence poset~\cite{Sch-89,Fel-01}, and are used to encode triangulations efficiently~\cite{Ber-09}. 
In particular, the local rule implies that every Schnyder wood gives an orientation of the inner edges such that every inner vertex has outdegree $3$ and the outer vertices are sources with respect to inner edges. Indeed, this is a one-to-one correspondance between Schnyder woods and orientations of this kind. As a consequence, the set of Schnyder woods of a planar triangulation inherits a natural distributive lattices structure, which in particular provides any planar graph with a unique \emph{minimal} Schnyder wood~\cite{Fel-04}. These unique representatives are an important tool in proofs and lie at the heart of many enumerative results, see for instance~\cite{Ber-07}.

When generalizing Schnyder woods to higher genus one has to choose which of the properties of planar Schnyder woods are desired to be carried over to the more general situation. Examples are: the efficient encoding of triangulations on arbitrary surfaces~\cite{Cas-09} and  the relation to orthogonal surfaces and small grid drawings for toroidal triangulations~\cite{Gon-14}, which lead to different definitions of generalized Schnyder woods. In~\cite{Gon-14}, the generalized Schnyder woods indeed satisfy the local rule with respect to \emph{all} edges and vertices of a toroidal triangulation and henceforth lead to orientations having outdegree $3$ at every vertex.
An interesting open problem is to generalize the local rule to
triangulations with higher Euler genus in such a way that for some
vertices the sequence mentioned in the local rule occurs several times around
the vertex. Here, the mere existence of such objects is an open question. Clearly, such a generalized Schnyder wood would yield an orientation as claimed by the conjecture. Thus, proving the conjecture of Bar\'at and Thomassen is a first step into that direction.

\section{Preliminaries}

A map $M$ on a surface $\mathbb{S}$ is characterized by a triple
$(V(M),E(M),F(M))$, formed by the vertex, edge and face sets of
$M$. In the following we will restrict to triangulations $T=(V(T),E(T),F(T))$, i.e. the pair $(V(T),E(T))$ is a simple embedded graph such that every face is incident to exactly three edges.

A \emph{submap} $M'$ of $T$, is a triplet $(V',E',F')$ where
$V'\subseteq V(T)$, $E'\subseteq E(T)$ and $F'\subseteq F(T)$. Note that a submap is not a map. 
A submap $M'=(V',E',F')$ is \emph{closed} if:
\begin{itemize}
\item[-] $uv \in E'$ implies $\{u,v\} \subseteq V'$, and
\item[-] $f \in F'$ implies $e'\in E'$ for any edge $e$ incident to $f$.
\end{itemize}
The \emph{closure} $cl(M')$ of a submap $M'$ (of $T$) is the smallest
closed submap of $T$ containing $M'$. The \emph{boundary} $\partial
M'$ of a submap $M'$ is the set of edges in $cl(M')$ that are incident
to at most one face in $M'$.

In a submap $M'$ of $T$ a \emph{(boundary) angle} $\widehat{e_0ve_t}$
at vertex $v$ is an alternating sequence $(e_0,f_1, e_1, \ldots,
f_t,e_t)$, for some $t\ge 1$, of edges and faces incident to $v$ (in
$T$) and such that:
\begin{itemize}
\item[-] the faces $f_i$ are mutually different, for $1\le i\le t$,
\item[-] each face $f_i$, for $1\le i\le t$, is incident to edges $e_{i-1}$ and $e_i$,
\item[-] both edges $e_0$ and $e_t$ belong to $cl(M')$,
\item[-] but none of the remaining edges, $e_i$ for $0<i<t$, belong to
  $cl(M')$, nor any faces $f_i$, for $0<i\le t$.
\end{itemize}
Here the angles we consider are not directed: the sequence
$(e_t,f_t,\ldots,e_1,f_1,e_0)$ defines the same angle as $(e_0,f_1,e_1,\ldots,f_t,e_t)$. Note that in $(e_0,f_1,e_1,\ldots,f_t,e_t)$ it can occur that $e_0=e_t$. Consider for example a submap consisting of a single edge. Let us mention, that this definition could be modified in order to include the angle around a vertex with respect to a submap without edges. Since we will not consider this situation we prefer avoiding further technicalities.


The notion of angles endows the boundary $\partial M'$ of $M'$ with some further structure. Note that an edge is in $\partial M'$ if and only if it is involved in (at least) one angle of $M'$. This
fact leads to the definition of the following relation. Two angles
$(e_0,f_1,e_1,\ldots,f_t,e_t)$ and $(e'_0,f'_1,e'_1,\ldots,f'_t,e'_t)$
are \emph{consecutive} on the boundary if $e_t=e'_0$ and
$f_t=f'_1$. As each angle $(e_0,f_1,e_1,\ldots,f_t,e_t)$ has two \emph{sides}: $e_0$ and $e_t$ , this relation leads to the
definition of \emph{boundary sequence}, that is a collection of circular sequences
of angles. Such a circular sequence is sometimes denoted by an
alternating sequence $(\hat{a_0},e_0,\hat{a_1},e_1,\ldots
,\hat{a_t},e_t)$ or simply by a sequence $(e_0,e_1,\ldots ,e_t)$,
where $e_i$ is the common edge of $\hat{a_i}$ and $\hat{a_{i+1}}$. Note that an edge $e$ may appear twice in the boundary sequence, e.g. if $e$ is a bridge of $M'$. Thus, if necessary we will refer to a specific \emph{occurrence} of $e$ in  $\partial M'$.
For simplicity, we denote the boundary sequence of $M'$ by $\partial M'$. 

In the following, a \emph{disk} is a submap $M'$ of $T$ if it
is homeomorphic to an (open or closed) topological disk.  Furthermore,
a disk is a \emph{$k$-disk} if its boundary is a cycle with $k$
edges. A $3$-disk is called \emph{trivial} if it contains only one
face. A disk is called \emph{chordless} if its outer vertices (i.e. on
its boundary) induce a graph that is a (chordless) cycle. A cycle is \emph{contractible} if it is the boundary of a disk otherwise it is called \emph{non-contractible}.

Given a triangulation $T$ and a set of vertices $X\subseteq V(T)$, the
\emph{induced submap} $T[X]$ is the submap with vertex set $X$, edge
set $\{uv\in E(T)\ |\ u\in X \mbox{ and } v\in X\}$, and face set
$\{uvw\in F(T)\ |\ u\in X, v\in X, \mbox{ and } w\in X\}$. Note that
induced submaps are always closed submaps.

Given an induced submap $M'=T[X]$ of a triangulation $T$, and any
occurrence of an edge $ab$ in $\partial M'$ (corresponding to angles
$\hat{a}$ and $\hat{b}$) there exists a unique vertex $c$ such that
there is a face $abc$ in $T\setminus M'$ that belongs to both angles
$\hat{a}$ and $\hat{b}$.  For any such vertex $c$ (and $ab \in
\partial M'$) we define the operation of \emph{stacking} $c$ on $M'$,
as adding $c$ to $X$, i.e., going from $M'=T[X]$ to $M''=T[X+c]$. In
such stacking, let $M'\cap cl(M''\setminus M')$ be the
\emph{neighborhood} of $c$ in $M'$. As $T$ is simple, note that this
neighborhood is either a cycle or a union of paths, one of which with
at least one edge (the edge allowing the stacking), and let us
respectively call them the \emph{neighboring cycle} and the
\emph{neighboring paths} of $c$ in $M'$. See Figure~\ref{fig:stack}
for an illustration.

 \begin{figure}[htb]
 \begin{center}
 \includegraphics[width=\textwidth]{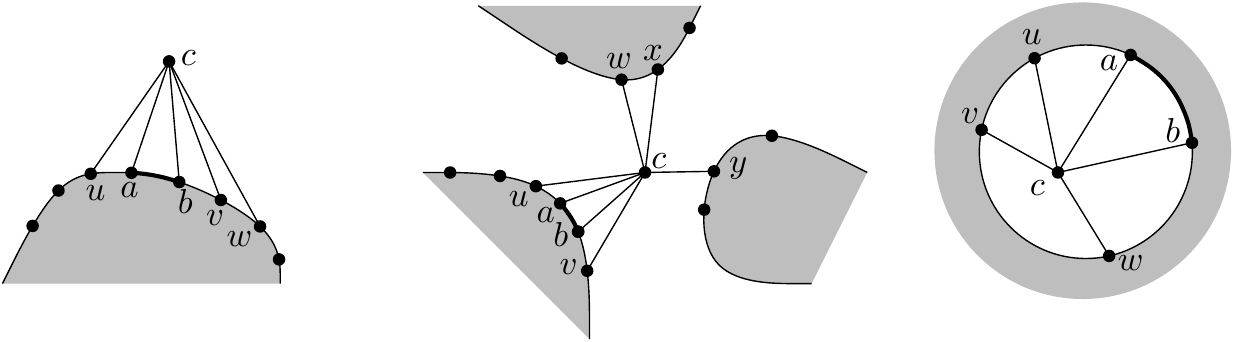}
 \caption{Different scenarios of stacking $c$ to $M'$. Left: one neighboring path $P_1=(u, a, b, v, w)$. Middle: three neighboring paths $P_1=(u, a, b, v)$, $P_2=(w, x)$, $P_3=(y)$. Right: A boundary cycle $C=(u, v, w, b, a)$.}
 \label{fig:stack}
 \end{center}
 \end{figure}

\section{Proof of Conjecture~\ref{conj}}

Let us consider for contradiction a minimal counterexample $T$. Note
that $T$ does not contain any non-trivial $3$-disk $D'$. Otherwise
we would remove the interior of $D'$ and would replace it by a face.
By minimality of $T$, this new triangulation would admit an
orientation such that every vertex has non-zero outdegree divisible by $3$. As $D'$ is a planar triangulation, there exists an
orientation of its interior edges so that inner and outer vertices
have respectively out-degree $0$ and $3$. This is the case for
orientations induced by a Schnyder wood on these
triangulations~\cite{Sch-89}.  Then the union of these two
orientations would give us an orientation of $T$ with non-zero outdegrees divisible by three.
Let us now proceed by providing an outline of the proof.

\subsection{Outline}

We first prove that one can partition the edges of the triangulation
$T$ into the following graphs:
\begin{itemize}
 \item The \emph{initial graph $I$}, which is an induced submap 
   containing a non-contractible cycle. Furthermore, $I$ contains an edge
   $e^*= \{u,v\}$ such that the map $I\setminus e^*$ is a disk $D$ whose
   underlying graph is a maximal outerplanar graph with only two degree
   two vertices, $u$ and $v$. See Figure~\ref{fig:I_sans-arcs} for an illustration.

 \begin{figure}[htb]
 \begin{center}
 \includegraphics[width=.9\textwidth]{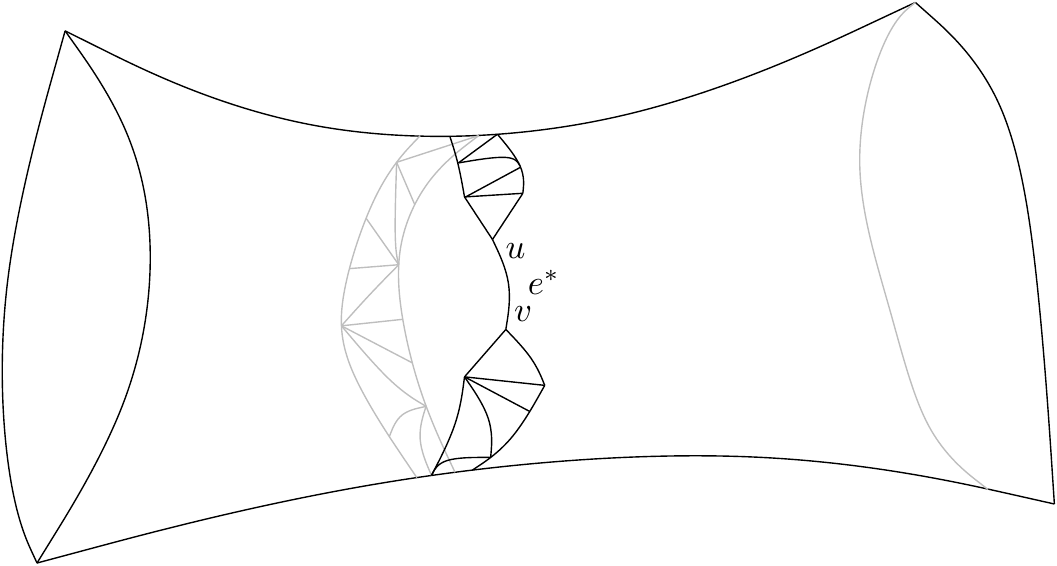}
 \caption{Example of a submap $I$.} 
 \label{fig:I_sans-arcs}
 \end{center}
 \end{figure}
 
 \item The \emph{correction graph $B$} (with blue edges in the
   figures), which is oriented acyclically in such a way that
   each vertex of $V(T)\setminus V(I)$ has outdegree 2, while
   the other vertices have outdegree $0$,

 \item The \emph{last correction path $G$} (with green edges in the
   figures), which is a $\{u,v\}$-path.


 \item The \emph{non-zero graph $R$} (with red edges in the figures),
   which is oriented in such a way that all vertices in
   $\left(V(T)\setminus V(G)\right)\cup\{u,v\}$ have out-degree at least $1$.
\end{itemize}
The existence of such graph $I$ is proven in Section~\ref{sub:Iexist}, then
in Section~\ref{sub:BGRexist} we prove the existence of graphs $B$, $G$ and
$R$ (with the mentioned orientations). To do the latter we start from
$I$ and we incrementally conquer the whole triangulation $T$ by
stacking the vertices one by one (this is inspired by~\cite{Cas-09}).

Finally, the edges of $I$, $B$ and $G$ are (re)oriented, to obtain the
desired orientation.  The orientation of edges in $R$ does not change,
as they ensure that many vertices (all vertices of $T$ except the
interior vertices of the path $G$) have non-zero outdegree. The
$\{u,v\}$-path $G$ is either oriented from $u$ to $v$ or from $v$ to
$u$, but this will be decided later. However in both cases its
interior vertices are ensured to have non-zero outdegree. Hence all
vertices are ensured to have non-zero outdegree and it remains to
prove that they have outdegree divisible by $3$.

We start in Section~\ref{sub:orientallexceptIG} by reorienting the
$B$-arcs in order to ensure that vertices of $V(T)\setminus V(I)$ have
outdegree divisible by $3$ (as in~\cite{Bar-06}). In the last step,
in Section~\ref{sub:orientIG}, we choose the orientation of the $\{u,v\}$-path
$G$, and we orient $I$ in order to achieve the desired orientation.


\subsection{Existence of $I$}\label{sub:Iexist}

To prove the existence of $I$, we first need the following lemma.

\begin{lemma}\label{lem:I1}
Any triangulation $T$ with Euler genus at least 2, has an induced
submap ${I}$ obtained from a disc ${D}$ by
stacking a vertex $v$, such that for any two neighbors $a, b$ of $v$
belonging to distinct neighboring paths (of $v$
w.r.t. ${D}$), every cycle $C$ in ${I}$ going
through edges $av$ and $vb$ is non-contractible.
\end{lemma}
\begin{proof}
Any face of $T$ is an induced disk. Consider a maximal induced disk ${D}$ of $T$.  For any edge $ab$ of
$\partial {D}$, stack a vertex $v$ to $xy$. Let us denote by ${I}$ the map obtained by stacking $v$ to ${D}$. As $T$ has Euler genus at least $2$ the neighborhood of $v$ is not a cycle. Also, as ${D}$ is maximal, $v$ has at least two neighboring paths. 
Assume for contradiction, that there is a contractible cycle $C$ of ${I}$
going through $av$, $vb$ (where $a$ and $b$ belong to distinct
neighboring paths of $v$ w.r.t. ${D}$) and through some $(a,b)$-path $P$
of $\partial {D}$. Denote ${D}'$ the disk bounded by $C$ and note that
(as ${I}$ is induced) ${D}'$ contains vertices not in ${I}$. Now it is clear
that $V({D})\cup \Int({D}')$ (i.e. $(V({D})\cup V({D}'))\setminus \{v\})$
induces a larger disk, contradicting the maximality of ${D}$.\hfill\qed
\end{proof}

%
%

\begin{lemma}\label{lem:I2}
  Any triangulation $T$ with Euler genus at least 2, has an submap $I$ containing an non-contractible cycle, and an edge $e^*=
  \{u,v\}$ such that $I\setminus e^*$ is a disk ${D}$, and for
  each of the two $(u,v)$-paths of $\partial {D}$, all its interior vertices
  have a neighbor in the interior of the other $(u,v)$-path.
\end{lemma}
\begin{proof}
Among the induced subgraphs of $T$ that satisfy Lemma~\ref{lem:I1}
let $I$ be a minimal one. Let respectively $v$ be a vertex of $I$, and $\widetilde{D}$
be the disk $I\setminus \{v\}$ described in Lemma~\ref{lem:I1}. As $v$ is stacked on $\widetilde{D}$ let us denote
$(w_1,\ldots,w_s)$, with $s\ge 2$, some neighboring path of $v$, and
let us denote $u_1,\ldots, u_t$, with $t\ge 1$, the other neighbors of
$v$ in $\widetilde{D}$. Finally, let us denote $D$ the disk obtained from $\widetilde{D}$ by
adding vertex $v$, edges $vw_i$ for $1\le i\le s$, and faces
$vw_iw_{i+1}$ for $1\le i <s$. The minimality of $I$ implies all the
needed properties:

\begin{myclaim}\label{cl:I1}
  $\partial \widetilde{D}$ induces no chord $xy$ inside $\widetilde{D}$ such that some $(x,y)$-path of
  $\partial \widetilde{D}$ contains both an edge $w_iw_{i+1}$, for some $1\le
  i<s$, and a vertex $u_j$, for some $1\le j\le t$.
\end{myclaim}
\begin{proof}
If such chord $xy$ exists, let $D'\subsetneq \widetilde{D}$ be the disk with
boundary in $\partial \widetilde{D}+xy$ which contains both $w_iw_{i+1}$ and
$u_j$.  Then the graph induced by $V(D')\cup\{v\}$ contradicts the minimality
of $I$.\hfill\qed
\end{proof}
This implies that $\partial D$ has no chord at $u_j$, for all $1\le
j\le t$.

\begin{myclaim}\label{cl:I2}
  For all $1\le j\le t$, every interior vertex $x$ of a $(v,u_j)$-path
  of $\partial D$ is adjacent to an interior vertex of the other
  $(v,u_j)$-path.
\end{myclaim}
\begin{proof}
Let $P_1$ and $P_2$ be the $(v,u_j)$-path of $\partial D$ containing
respectively $w_1$ and $w_s$. Assume for contradiction, there exists
an inner vertex $x$ in $P_1$ having no neighbor in the interior of
$P_2$. By Claim~\ref{cl:I1} this implies that $\widetilde{D}$ (the map
induced by $V(I)\setminus \{v\}$) has no chord at $x$. Thus the map
induced by $V(\widetilde{D})\setminus\{x\}$ is a disk still containing
the vertex $u_j$ and the edge $w_{s-1}w_s$ on its border. Hence the
map induced by $V(I)\setminus\{x\}$ contradicts the minimality of
$I$.\hfill\qed
\end{proof}
As $\partial D$ has no chord at $u_j$, for all $1\le j\le t$, this
implies that $t=1$. This concludes the proof of the lemma.\hfill\qed
\end{proof}

In the beginning of the proof we have seen that by minimality, $T$
does not contain non-trivial $3$-disks. Hence by the properties of
$I$, if $D$ ($=I\setminus e^*$) would contain an inner vertex, this
vertex would be in a chordless $4$-disk of $D$. By the following
lemma this is not possible, hence $D$ is a maximal outerplanar
graph. Finally the adjacency property between vertices of $\partial
D\setminus\{u,v\}$ imply that $u$ and $v$ are the only degree two
vertices of $D$.

\begin{lemma}
The submap $D$ does not contain chordless $4$-disks.
\end{lemma}
\begin{proof}
If $D$ would contain such a disk $D'$, with boundary
$(v_1,v_2,v_3,v_4)$, we would remove the interior of $D'$ and we would
add one of the two possible diagonals, say $v_2v_4$ (if $v_2v_4$ ar not $e^*$'s ends), and the corresponding two triangular faces, $v_1v_2v_4$ and
$v_2v_3v_4$. The obtained map $T'$ is defined on the same surface as
$T$ and is smaller. Furthermore as $I$ is an induced submap without
non-trivial $3$-disk and as $v_2v_4 \neq e^*$, there was no edge
$v_2v_4$ in $T$.  Hence $T'$ is simple and it is a triangulation. Now
by minimality of $T$, this new triangulation $T'$ has an
orientation such that every vertex has non-zero outdegree divisible by $3$. Let us suppose without loss of generality that in this
orientation the edge $v_2v_4$ is oriented from $v_2$ to $v_4$.

Using the fact that for any planar triangulation, there exists an
orientation of the interior edges such that inner and outer vertices
have respectively out-degree $0$ and $3$~\cite{Sch-89}, one can orient the
inner edges of $D'$ in such a way that inner vertices, vertex $v_2$,
and vertices $v_1$, $v_3$ and $v_4$ have respectively out-degree $3$, $1$
and $0$. For this consider the orientation of the triangulation
$D'+v_1v_3$, with outer face $v_1v_3v_4$, and notice that the edges
$v_2v_1$ and $v_2v_3$ are necessarily oriented from $v_2$ to $v_1$ and
$v_3$ respectively (as $v_1$, $v_3$ and $v_4$ have outdegree $0$).

Then the union of these orientations, of $T'\setminus v_2v_4$ and of
$D'$'s inner edges, would give us an orientation of $T$ with non-zero
outdegrees divisible by three.\hfill\qed
\end{proof}

\subsection{Existence of $B$, $G$, and $R$}\label{sub:BGRexist}

As mentioned in the outline, we will start from $I$ and we
incrementally explore the whole triangulation $T$ by stacking the
vertices one by one. At each step, we will assign the newly explored
edges to $B$, $G$ or $R$, and we will orient those assigned to $B$ or
$R$.  At each step the \emph{explored part} is a submap of $T$ induced
by some vertex set $X$. The explored part is hence the submap denoted
$T[X]$ with boundary $\partial T[X]$. The \emph{unexplored part} is the
submap $T\setminus T[X]$, and it may consist of several components.

At a given step of this exploration, the graph $G$ may not be an
$\{u,v\}$-path yet.  In such a case, the graph $G$ will consist of two
separate paths $G_u$ and $G_v$, respectively going from $u$ to $u'$,
and from $v$ to $v'$, for some vertices $u'$ and $v'$ on $\partial
T[X]$. Here the vertices $u'$ and $v'$ may respectively coincide with
vertices $u$ and $v$, if $G_u$ or $G_v$ is a trivial path on just one
vertex. In such a case, the vertices $u'$ and $v'$ are called the
\emph{current ends} of $G$.

During the exploration we maintain the following invariants:
\begin{itemize}
 \item[(I)] The graphs $I$, $B$, $G$, and $R$ partition the edges of $T[X]$.
 \item[(II)] All interior vertices of $T[X]$ have at least one outgoing
   $R$-arc, or two incident $G$-edges.
 \item[(III)] The graph $B$ is acyclically oriented in such a way that
   the vertices of $I$ have outdegree 0, while the other vertices of
   $T[X]$ have outdegree 2.
 \end{itemize}
Furthermore, to help us in properly finishing the construction of the
graphs $B$, $G$ and $R$ in the further steps, we introduce the notion
of \emph{requests} on the angles of $\partial T[X]$. Informally, a
$G$-request (resp. an $R$-request) for an angle $\hat{a}$ means that
in a further step an edge inside this angle will be added in $G$
(resp. in $R$ and oriented from $a$ to the other end). Every angle has
at most one request, and an angle having no request is called
\emph{free}.
 \begin{itemize}
 \item[(IV)] Every vertex of $(\partial T[X] \setminus \{u',v'\})
   \cup\{u,v\}$ having (still) no outgoing $R$-arc, has an incident
   angle with an $R$-request.
 \item[(V)] If $G$ is not a $\{u,v\}$-path (yet), its current ends, $u'$ and $v'$,
   have one incident angle each, say $\hat{u}'$
   and $\hat{v}'$, that are consecutive on $\partial T[X]$, and that
   have a $G$-request.
 \item[(VI)] If there is an unexplored disk $D'$, i.e. a component of
   the unexplored part that is a disk, then there are at least three
   free angles (of $\partial T[X]$) around $D'$.
\end{itemize}

This exploration starts with $T[X]=I$. In this case as all the edges
of $T[X]$ are in $I$ and as there are no interior vertices yet, (I),
(II) and (III) are trivially satisfied. Since the Euler genus of $T$ is
at least $2$ there is no unexplored disk, hence (VI) is satisfied.
Since $e^*=uv$ appears twice in $\partial T[X]$, the vertices $u,v$
appear twice consecutively in $\partial T[X]$. To achieve (V), choose
the angles of one consecutive appearance of $u,v$ as $G$-requests. To
achieve (IV), all the other angles are assigned $R$-requests. See
Figure~\ref{fig:I} for an illustration.

 \begin{figure}[htb]
 \begin{center}
 \includegraphics[width=.9\textwidth]{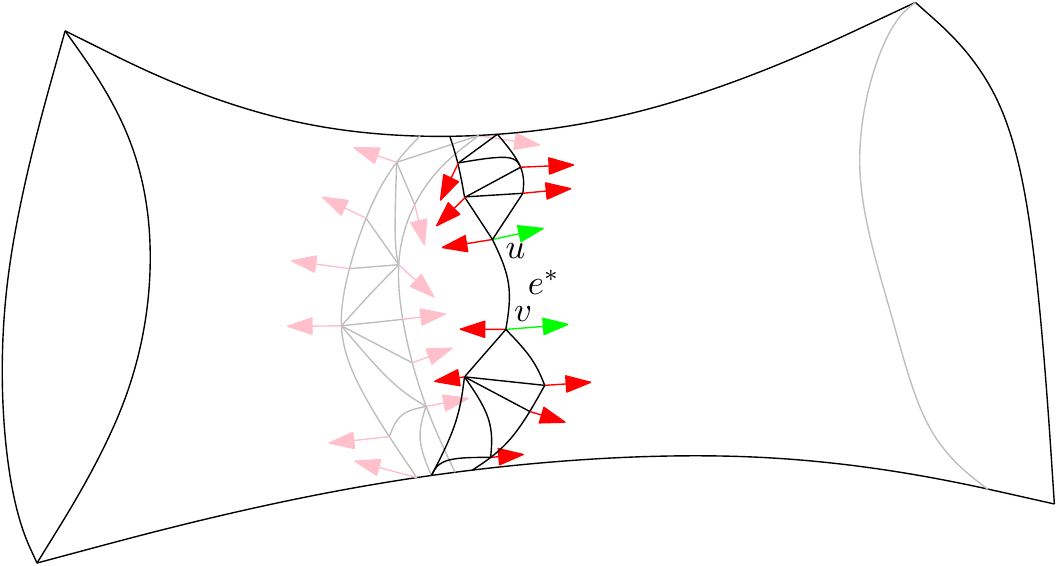}
 \caption{Assigning requests to $I$ in order to satisfy the invariants.} 
 \label{fig:I}
 \end{center}
 \end{figure}

For the rest of the construction in each step we enlarge the explored map $T[X]$ by stacking a vertex $x$ to $e\in\partial T[X]$. The vertex $x$ is chosen according to the following rules:
\begin{itemize}
\item[(i)] Not both ends of $e$ have a $G$-request.
\item[(ii)] If $x$ belongs to an unexplored disk $D$, either $x$ is adjacent
  to all vertices of $\partial D$ or $x$ has exactly one neighboring
  path $P$ on $\partial D$ such that $P$ does not contain all the
  free angles of $\partial D$.
\item[(iii)] In the case $x$ does not belong to an unexplored disk, if
  possible we choose $x$ such that no unexplored disk is created.
\end{itemize} 
Using that $T$ has no non-trivial $3$-disks, one can easily check that choosing such a vertex $x$ is always possible.
%

In the following we show how to extend $B, G, R$ on the new introduced
edges and how to deal with the newly created angles to maintain all
invariants valid. We will describe the construction and we will
check the validity of invariants only for the non-trivial ones.
We distinguish cases according to the topology of the unexplored
region containing $x$.\\

\noindent
{\bf 1) The vertex $x$ is contained in an unexplored disk $D$ and has a
  neighboring cycle.} By (VI) the unexplored disk containing $x$ has
at least $3$ free angles. We orient the corresponding edges from $x$
to its neighbors, put two into $B$ and the rest into $R$. All non-free
angles satisfy their request with the edge incident to $x$.

We have assigned all the newly explored edges, hence (I) remains
valid.  As (IV) and (V) were valid in $T[X]$, all the neighbors of $x$
(i.e. the vertices around $D$) have now (in $T[X+x]$) an outgoing
$R$-arc or two incident $G$-edges. The vertex $x$ also does, hence
(II) is valid.  In the acyclic graph $B$, adding the vertex $x$ with
only outgoing $B$-arcs cannot create any circuit, hence (III) remains
valid. As in this case, as $\partial T[X+x]$ is included in $\partial
T[X]$, (IV) remains valid. If $u'$ and $v'$ were around $D$ in $T[X]$,
the two parts of $G$ are now connected by the adjunction of $xu'$ and
$xv'$ in $G$. Otherwise, $G$ was already an $\{u,v\}$-path, or $u'$
and $v'$ were elsewhere in $\partial T[X]$ fulfilling (V). Hence in
any case (V) holds. Finally, as no unexplored disk has been created
and as the requests around existing unexplored disks have not changed,
(VI) remains valid.

\smallskip

For the remaining cases we introduce some further notation. Given a
neighboring path $P = (p_1,\ldots,p_s)$ of $x$, with corresponding
angles $\hat{p_1},\ldots,\hat{p_s}$, the {\em inner angles} are the
angles $\hat{p_i}$ with $1<i<s$. The other ones are the \emph{outer angles}. 
An inner angle with an $R$- or
$G$-requests, has to satisfy its constraint (this cannot be further
delayed). Hence for any inner angle $\hat{p_i}$ with an $R$-request
(resp. a $G$-request) we add the edge $xp_i$ to $R$ (resp. to $G$)
oriented towards $x$. This is a preprocessing step valid for both the remaining two cases.\\

\noindent
{\bf 2) The unexplored region containing $x$ is not a disk.}  For
simplicity assume, that there are no free angles. Otherwise we assign
an $R$-request to all these angles. Here after the preprocessing step described above, there is an intermediate step 2.1) and a final step 2.2). See Figure~\ref{fig:notinadisk} for an illustration of how this case is handled.\\

 \begin{figure}[htb]
 \begin{center}
 \includegraphics[width=\textwidth]{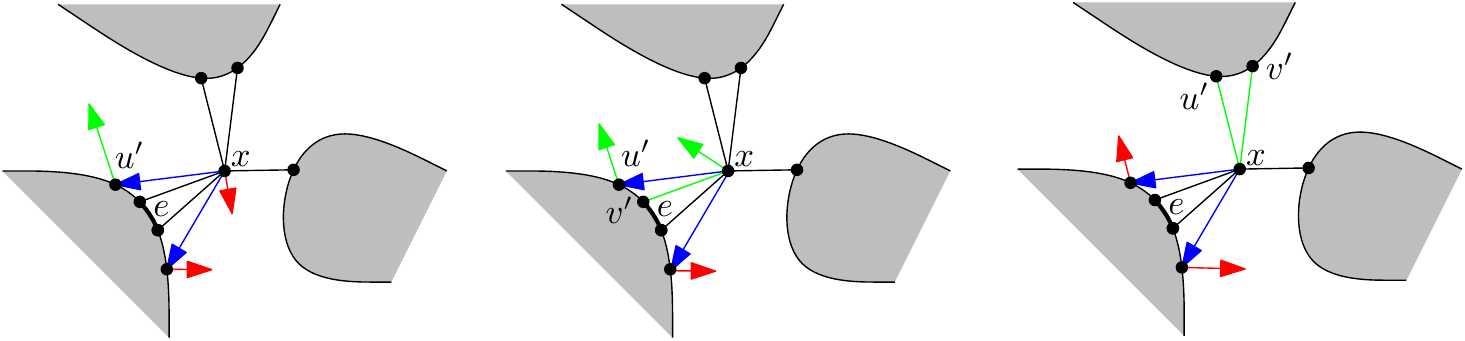}
 \caption{Dealing with the case that $x$ is not in an unexplored disk.} 
 \label{fig:notinadisk}
 \end{center}
 \end{figure}

\noindent
{\bf 2.1) The intermediate step.} This step depends on the position of
the $G$-requests, if any.

\noindent
{\bf If there is no $G$-request on the neighboring paths of $x$}, then we
assign an $R$-request to some angle $\hat{x}$ incident to $x$.

\noindent
{\bf If only one $G$-request (say on $\hat{u}'$) is on a neighboring path
of $x$}, then $\hat{u}'$ is an end of this neighboring path.
Here the new angle at $u'$ (inside the former angle $\hat{u}'$) that is created by stacking
$x$ inherits $\hat{u}'$'s $G$-request. If two angles are created inside the former angle $\hat{u}'$, that is if $u'$ is alone in its neighboring path, we choose the angle next to $v'$ in order to fulfill (V). Then we assign an $R$-request to some angle $\hat{x}$ incident to $x$.

\noindent
{\bf If one $G$-request say $\hat{u}'$ is on an outer angle and the other
one $\hat{v}'$ on an inner one}, we have added the edge $v'x$ to $G$ in the preprocessing.
Here the new angle at $u'$ inherits $\hat{u}'$ 's $G$-request and the next angle on $\partial T[X+x]$, that is incident to $x$ gets a $G$-request too.

\noindent
{\bf If both $G$-requests are on inner angles}, the edges $v'x$ and $u'x$ have
been added to $G$ in the preprocessing. Hence $x$ has already two incident $G$-edges and does not need any request around. We thus leave all angles incident to $x$ free.

\noindent
{\bf If both of the $G$-requests are on outer angles}, then by (i) $x$ has
one length one neighboring path $(u',v')$, and at least one other
neighboring path of length at least one. In that case, we add edges
$v'x$ and $u'x$ to $G$ and we leave the new angles at $u'$ and $v'$, as well 
as all angles incident to $x$, free.\\

\noindent
{\bf 2.2) The final step.} We now assign two outgoing $B$-arcs to $x$, depending on the
$G$-requests. If there is an outer angle $\hat{u}'$ (in $T[X+x]$) with a $G$-request
add the arc $xu'$ directed towards $u'$ to $B$. The remaining one or
two needed $B$-arcs are chosen arbitrarily among the edges from $x$
to outer vertices. All other edges, between $x$ and outer vertices
will be put into $R$ and directed towards $x$. Note that among the
newly created outer angles and the angles associated to $x$ there are
at most $3$ requests: two at the angles receiving a $B$-arc from $x$
and one at an angle incident to $x$.

In case, that adding $x$ creates an unexplored disk $D'$, we still have to argue, that (VI) is satisfied with respect to $D'$. We make use of the following:

\begin{myclaim}\label{cl:disks} 
For any unexplored disk $D'$ created by stacking a vertex $x$ on
$T[X]$, the vertex $x$ appears several times on the boundary of $D'$.
\end{myclaim}
\begin{proof}
Suppose we create an unexplored disk $D'$ such that $x$ appears only
once on its boundary. Assume $x$ is chosen such that the number of
faces in $D'$ is minimized. Since there are no non-trivial $3$-disks the boundary of $D'$ is of length at least $4$. Therefore $D'$
contains an unexplored vertex $x'$ which could have been stacked to a
subsequence of $\partial D' \setminus x$. Furthermore, $x'$ can be
chosen such that the sequence not only contains the
$G$-requests. Hence, stacking $x'$ would either not create any
unexplored disk, or would create a smaller one, both contradicting the
choice of $x$.\hfill\qed
\end{proof}

This claim and the fact that $T$ is simple imply that there are at
least $6$ angles on the boundary of $D'$ incident to outer vertices of
the neighborings paths of $x$ (4 of them) or incident to $x$ (2 of
them).  As argued above at most $3$ of these angles have a
request. Thus, there are at least $3$ free angles on the boundary of
$D'$ and (VI) is satisfied.\\

\noindent
{\bf 3) The unexplored region containing $x$ is a disk, but $x$'s
  neighborhood is not a cycle.} By (ii) the vertex $x$ has only one
neighboring path. Let us denote this path by $P = (p_1,\ldots,p_s)$
for some $s\ge 2$ and $\hat{p_1}, \ldots, \hat{p_s}$ the corresponding
angles.  Denote by $t$ the number of free angles on $P$.  

We start with the preprocessing described above, that deals with
non-free interior angles (by fulfilling the requests).  To fulfill (VI) we have to
maintain the number of free angles in this unexplored disk above
three. Since by (ii) there is at least one free angle not on $P$, to
achieve this we need to have at least $\min \{t,2\}$ free angles among
the new angles $\hat{p_1}$, $\hat{x}$, and $\hat{p_s}$.

To achieve that we need to exploit free angles as follows. For any
free angle $\hat{p_i}$ (inner or not), the edge $xp_i$ is added either
to $B$ or to $R$, in both cases oriented towards $p_i$. Among these
$t$ angles, $\min\{t,2\}$ lead to a $B$-arc, and $\max\{0,t-2\}$ lead
to an $R$-arc.  It remains to deal with angles that are neither inner
nor free.  We proceed by distinguishing cases according to the
position of $G$-requests.

\noindent
{\bf If there is no $G$-request on $P$,} then if $t\le 2$ we assign an
$R$-request to the angle $\hat{x}$ and otherwise we leave $\hat{x}$ free.
Then we use $\max\{0,2-t\}$ of the non-free outer angles to add $B$-arcs
leaving $x$.  We satisfy the possibly remaining non-free outer angles (that are $\min\{2,t\}$),
by adding $R$-arc towards $x$, and leave their new incident angle free.
If $t\le 2$ (resp. $t\ge 3$), there are $\min\{2,t\}=t$ (resp. $1+\min\{2,t\}=3$)
free angles among the new angles $\hat{p_1}$, $\hat{x}$, and $\hat{p_s}$.

\noindent
{\bf If only one $G$-request (say on $\hat{u}'$) is on $P$,} then
$\hat{u}'$ is an end of $P$, say $p_1=u'$ (see
Figure~\ref{fig:notinadisk2}). Here the new angle at $u'$ inherits
$\hat{u}'$'s $G$-request, and we add the edge $xp_1$ in $B$ if $t\le
1$, or in $R$ otherwise. In both cases $xp_1$ is oriented towards
$p_1$.  If $t=0$ then $\hat{p_s}$ is not free (i.e. $\hat{p_s}$ has an
$R$-request), but we add $xp_s$ in $B$ oriented from $x$ to $p_s$ and
the new angle $\hat{p_s}$ inherits the $R$-request.  If $t\ge 1$, we
satisfy the $R$-request of $\hat{p_s}$ (if it has one) with edge
$xp_s$.  In any case, $\hat{p_s}$ having a request or not, the new
angle $\hat{p_s}$ is left free. If $t\le 1$ we assign an $R$-request
to angle $\hat{x}$ and otherwise we leave $\hat{x}$ free. Hence if
$t\ge 1$ the angle $\hat{p_s}$ is free, and if $t\ge 2$ the angle
$\hat{x}$ is free.

 \begin{figure}[htb]
 \begin{center}
 \includegraphics[width=\textwidth]{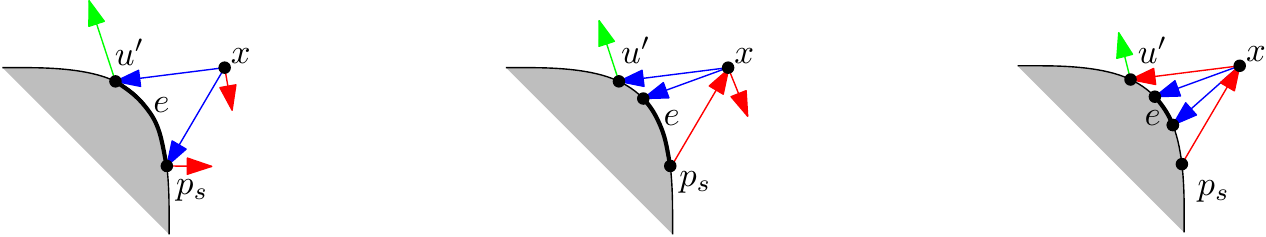}
 \caption{Case where there is only one $G$-request (on $\hat{u}'$) and where $\hat{p_s}$ has an $R$-request. The 3 subcases from left to right correspond to $t=0$, $t=1$, and $t=2$.}
 \label{fig:notinadisk2}
 \end{center}
 \end{figure}

\noindent
{\bf If one $G$-request say $\hat{u}'$ is on an outer angle and the other
one $\hat{v}'$ on an inner one,} say $u'=p_1$ and $v'=p_2$ with $s>2$,
we have added the edge $p_2x$ to $G$ (see Figure~\ref{fig:notinadisk3}). 
Around $p_1$, if $t\le 1$ we assign the new angles $\hat{p_1}$ and $\hat{x}$ a
$G$-request, and we add the edge $xp_1$ in $B$ oriented from $x$ to
$p_1$. Otherwise (i.e. $t\ge 2$) we add the edge $p_1x$ to $G$, and we
leave both new angles $\hat{p_1}$ and $\hat{x}$ as free.
Around $p_s$, if $t=0$ (hence $\hat{p_s}$ has an $R$-request) we add
$xp_s$ in $B$ oriented from $x$ to $p_s$, and the new angle
$\hat{p_s}$ keeps its $R$-request. Otherwise (i.e. $t\ge 1$), if
$\hat{p_s}$ has an $R$-request we add $xp_s$ in $R$ and orient it from
$p_s$ to $x$, and in any case ($\hat{p_s}$ having an $R$-request or
not) we leave the new angle $\hat{p_s}$ as free.
Hence if $t\ge 1$ the angle $\hat{p_s}$ is free, and if $t\ge 2$ both
$\hat{p_1}$ and $\hat{x}$ are free.

 \begin{figure}[htb]
 \begin{center}
 \includegraphics[width=\textwidth]{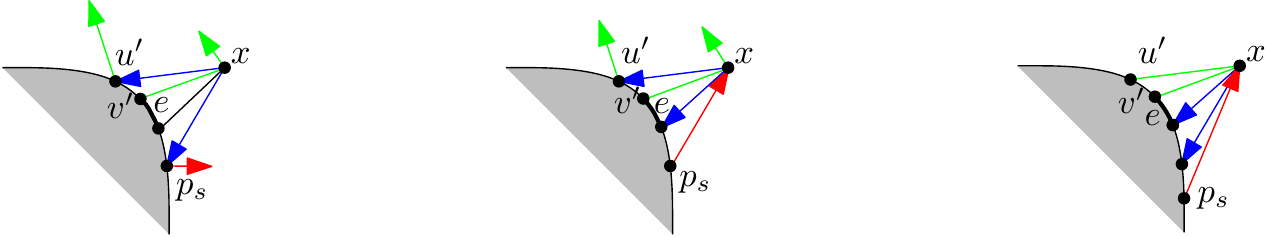}
 \caption{Case where there is one $G$-request on an outer angle, and one in an inner angle, and where $\hat{p_s}$ has an $R$-request. The 3 subcases from left to right correspond to $t=0$, $t=1$, and $t=2$.}
 \label{fig:notinadisk3}
 \end{center}
 \end{figure}

\noindent
{\bf If both $G$-requests are on inner angles,} edges $v'x$ and $u'x$
have been added to $G$ and now $x$ is an inner vertex of $G$. We thus
leave the angle incident to $x$ free.  Then we use $\max\{0, 2-t\}$ of
the outer angles for $B$-arcs from $x$, and the remaining non-free
outer angles have their $R$-requests satisfied, and are left free. In
any case, $\min\{2,t\}$ of the outer angles are free, as the angle
$\hat{x}$.

Finally by definition of stacking $P$, $x$'s unique neighboring path,
is distinct from $(u',v')$ and hence all cases have been addressed.

\subsection{Reorienting $B$}\label{sub:orientallexceptIG}
Given a partial orientation $O$ of $T$ we define the \emph{demand} of a vertex $v$ as $\dem_{O}(v):=-\delta^+_{|O}(v) \mod 3$, where $\delta^+_{|O}(v)$ denotes the outdegree of $v$ with respect to $O$. We want to find an orientation of $T$ with all demands~$0$.

Recall we will not modify the orientation on $R$, which guarantees that all vertices in $\left(V(T)\setminus V(G)\right)\cup\{u,v\}$ have non-zero outdegree. Furthermore, as $G$ will be oriented either entirely forward or backwards, all its interior vertices will have non-zero outdegree. Hence every vertex of $T[X]$ has non-zero outdegree. Suppose that $G$ is entirely oriented forward.

Now we linearly order vertices in $V(T)\setminus V(I)=(v_1, \ldots, v_{\ell})$ such that with respect to $B$ every vertex has its two outgoing $B$-neighbors among its predecessors and $I$. Denote by $B_i$ the subgraph of $B$ induced by the arcs leaving $v_i,\ldots, v_{\ell}$ (before the reorienting). We process $V(T)\setminus V(I)$ from the last to the first element. At a given vertex $v_i$ we look at $\dem_{G\cup R\cup B_i}(v_i)$ and reorient the two originally outgoing $B$-arcs of $v_i$ in such a way that afterwards $\dem_{G\cup R\cup B_i}(v_i)=0$ (i.e. $\delta^+_{|G\cup R\cup B_i}(v_i) \equiv 0 \mod 3$). As these $B$-arcs were heading at $I$ or at a predecessor, the demand on the vertices $v_j$, with $j>i$, is not modified and hence remains 0.

\subsection{Orienting $G$ and $I$}\label{sub:orientIG}


Denote by $O$ the partial orientation of $T$ obtained after~\ref{sub:orientallexceptIG}. Pick an orientation of $G$ (either all forward or all backward) and of $e^*=uv$ such that for the resulting partial orientation $O'$ we have $\dem_{O'}(v)\equiv 1\mod 3$. 

Now, take the triangle $\Delta$ of $I$ containing $v$. Since $D=I\setminus e^*$ is a
maximal outerplanar graph with only two degree two vertices, $D$ can
be peeled by removing degree two vertices until reaching $\Delta$.
When a vertex $x$ is removed orient its two incident edges so that
$\dem_{O'}(x)= 0$ (as for $B$-arcs). We obtain a partial orientation $O''$,
such that all vertices except the ones of $\Delta$ have non-zero
outdegree divisible by $3$.


Since the number of edges of $T$, and the number of edges of $\Delta$
are divisible by $3$, the number of edges of $T\setminus \Delta$ is
divisible by $3$. As this number equals the sum of the outdegrees in
$O''$, and as every vertex out of $\Delta$ has outdegree divisible by
3, then the outdegree of $\Delta$'s vertices sum up to a multiple of
3. Hence their demands sum up to 0, 3 or 6. As
$\dem_{O''}(v)=\dem_{O'}(v) = 1$, the demands of the other two
vertices of $\Delta$ are either both $1$, or $0$ and $2$. It is easy
to see that in either case $\Delta$ can be oriented to satisfy all
three demands.





\bibliographystyle{my-siam}
\bibliography{baratthomassen}

\end{document}